\documentclass[11pt]{amsart}

\usepackage[margin=1.25in]{geometry}

\usepackage{graphicx}

\usepackage{amsmath, amssymb, amsthm, graphicx, enumerate, tikz, hyperref, float}
\usepackage{mathtools, comment}
\usetikzlibrary{matrix,arrows,decorations.pathmorphing}

\newtheorem{theorem}{Theorem}[section]
\newtheorem{lemma}[theorem]{Lemma}
\newtheorem{corollary}[theorem]{Corollary}
\newtheorem{proposition}[theorem]{Proposition}
\newtheorem*{mainresult}{Main Theorem}

\theoremstyle{definition}
\newtheorem{definition}[theorem]{Definition}

\newtheorem*{solution*}{Solution}

\theoremstyle{remark}

\newtheorem*{pf*}{Pf}
\newtheorem*{pfbase*}{Proof of Base Case}
\newtheorem*{pfstep*}{Proof of Inductive Step}

\numberwithin{equation}{section}

\begin{document}

\allowdisplaybreaks

\title[Classifying Families of Character Degree Graphs of Solvable Groups]{Classifying Families of Character Degree Graphs of Solvable Groups}

\author{Mark W. Bissler}
\address{Department of General Education - College of Business, Western Governors University, Salt Lake City, Utah 84107}
\email{mark.bissler@wgu.edu}

\author{Jacob Laubacher}
\address{Department of Mathematics, St. Norbert College, De Pere, Wisconsin 54115}
\email{jacob.laubacher@snc.edu}


\date{\today}


\begin{abstract}
We investigate prime character degree graphs of solvable groups. In particular, we consider a family of graphs $\Gamma_{k,t}$ constructed by adjoining edges between two complete graphs in a one-to-one fashion. In this paper we determine completely which graphs $\Gamma_{k,t}$ occur as the prime character degree graph of a solvable group.
\end{abstract}

\maketitle

\section*{Introduction}

We note that throughout this paper, $G$ will be a finite solvable group. We denote Irr$(G)$ for the set of irreducible characters of $G$, and cd$(G)=\{\chi(1)\mid \chi\in\mbox{Irr}(G)\}$. Write $\rho(G)$ to be the set of primes that divide degrees in cd$(G)$. When working solely with graphs $\Gamma$ (and not necessarily degree graphs of a finite solvable group $G$), we will also use the notation $\rho(\Gamma)$ to signify the vertex set. The degree graph of $G$, written as $\Delta(G)$, is the graph whose vertex set is $\rho(G)$. Two vertices $p$ and $q$ of $\rho(G)$ are adjacent in $\Delta(G)$ if there exists $a\in\mbox{cd}(G)$ where $pq$ divides $a$. We identify each vertex of a graph with a prime in $\rho(G)$. Throughout this paper, for simplicity, when labeling a vertex we associate that label also with a prime $p\in\rho(G)$. Character degree graphs have been studied in a variety of places; for example, see \cite{H}, \cite{L}, \cite{Lewis}, \cite{L3}, \cite{Palfy}, and \cite{Sass}. We expand upon the work done in \cite{BL}.

Families of graphs have been considered and studied by way of direct products for some time. However, showing that a family of graphs cannot occur as the prime character degree graph of any solvable group was done recently in \cite{BL}. One of the main tools in that paper was to classify vertices as admissible. In particular, it was shown that the graph cannot occur when every vertex is admissible. Many arguments in this paper rely on this result.

In this paper, we construct a family of a graphs, denoted $\Gamma_{k,t}$, and inquire for which values of $k$ and $t$ does $\Gamma_{k,t}$ occur as the prime character degree graph of some solvable group. We construct $\Gamma_{k,t}$ by taking two complete graphs (uniquely determined by the number of vertices, of size $k$ and $t$, respectively), and we place an edge between the two graphs injectively. That is, we attach edges uniquely from one complete graph to the other in a one-to-one fashion. The construction clearly gives a graph that satisfies P\'alfy's condition and is at most diameter two. However, when determining if the graph can or cannot occur as the prime character degree graph of a solvable group, many of our arguments rely on facts about graphs that are diameter three. This has been studied more extensively in \cite{L} and \cite{Sass}, and their techniques and results will play an important role.

For $t\geq1$, we handle the cases separately for $k=1$ (Proposition \ref{early}), $k=2$ (Proposition \ref{bissresult}), $k=3$ (Proposition \ref{main}), and $k\geq4$ (Proposition \ref{huge}). We can sum up their parts in the result below:

\begin{mainresult}
The graph $\Gamma_{k,t}$ occurs as the prime character degree graph of a solvable group precisely when $k=1$, or $t=1$, or $k=t=2$.
\end{mainresult}

\section{Preliminaries}

Here we present some classic and more recent results that are immediately related to prime character degree graphs of solvable groups. Our aim is to keep this paper as self-contained as possible, and one can see \cite{I} or \cite{L3} for more reading. Most of the results from this section, however, can be found in \cite{BL}, \cite{Palfy}, or \cite{Sass}.

\begin{lemma}[P\'alfy's condition]\label{PC}\emph{(\cite{Palfy})}
Let $G$ be a solvable group and let $\pi$ be a set of primes contained in $\Delta(G)$. If $|\pi|=3$, then there exists an irreducible character of $G$ with degree divisible by at least two primes from $\pi$. (In other words, any three vertices of the prime character degree graph of a solvable group span at least one edge.)
\end{lemma}

\begin{definition}(\cite{BL})
A vertex $p$ of a graph $\Gamma$ is \textbf{admissible} if:
\begin{enumerate}[(i)]
    \item the subgraph of $\Gamma$  obtained by removing $p$ and all edges incident to $p$ does not occur as the prime character degree graph of any solvable group, and
    \item none of the subgraphs of $\Gamma$ obtained by removing one or more of the edges incident to $p$ occur as the prime character degree graph of any solvable group.
\end{enumerate}
\end{definition}

\begin{lemma}\label{mark}\emph{(\cite{BL})}
If $\Gamma$ is a graph in which every vertex is admissible, then $\Gamma$ is not the prime character degree graph of any solvable group.
\end{lemma}

Taking direct products is a method to generate graphs from smaller ones. We present the detailed construction below.

\begin{definition}
A graph $\Gamma$ is a \textbf{direct product} if it can be constructed in the following way:

We start with graphs $A$ and $B$. We have that $\rho(\Gamma)=\rho(A)\cup\rho(B)$, where $\rho(A)$ and $\rho(B)$ are disjoint sets. There is an edge between vertices $p$ and $q$ in $\Gamma$ if any of the following are satisfied:
\begin{enumerate}[(i)]
    \item $p,q\in\rho(A)$ and there is an edge between $p$ and $q$ in $A$,
    \item $p,q\in\rho(B)$ and there is an edge between $p$ and $q$ in $B$, or
    \item $p\in\rho(A)$ and $q\in\rho(B)$.
\end{enumerate}
\end{definition}

\begin{lemma}\label{complete}
Complete graphs occur as $\Delta(G)$ for some solvable group $G$.
\end{lemma}

The following two results are due to Lewis. They will aid in constructions throughout this paper.

\begin{lemma}\label{crazy}\emph{(\cite{L3})}
For every positive integer $N$, there is a solvable group $G$ so that $\Delta(G)$ has two complete, connected components: one having an isolated vertex, and the other having $N$ vertices.
\end{lemma}

\begin{lemma}\label{lewisgraph}\emph{(\cite{L})}
Let $G$ be a solvable group and let $p\in\rho(G)$. If $P$ is a normal Sylow $p$-subgroup of $G$, then $\rho(G/P')=\rho(G)\setminus\{p\}$.
\end{lemma}

The above Lemma \ref{lewisgraph} implies the subgraph $\Delta(G/P')$ is a subgraph of $\Delta(G)$ without the vertex $p$, all edges incident to $p$, and possibly without edges adjacent to vertices adjacent to $p$.

The following was presented in \cite{Sass}. Suppose $\Gamma$ is a graph of diameter three. We can partition $\rho(\Gamma)$ into four nonempty disjoint sets: $\rho(\Gamma)=\rho_1\cup\rho_2\cup\rho_3\cup\rho_4$. One can do this in the following way: find vertices $p$ and $q$ where the distance between them is three. Let $\rho_4$ be the set of all vertices that are distance three from the vertex $p$, which will include the vertex $q$. Let $\rho_3$ be the set of all vertices that are distance two from the vertex $p$. Let $\rho_2$ be the set of all vertices that are adjacent to the vertex $p$ and some vertex in $\rho_3$. Finally, let $\rho_1$ consist of $p$ and all vertices adjacent to $p$ that are not adjacent to anything in $\rho_3$. This notation is not unique, and depends on the vertices $p$ and $q$. Using this notation, we have the following result.

\begin{proposition}\label{sassresult}\emph{(\cite{Sass})}
Let $G$ be a solvable group where $\Delta(G)$ has diameter three. One then has the following:
\begin{enumerate}[(i)]
    \item\label{sass1} $|\rho_3|\geq3$,
    \item\label{sass2} $|\rho_1\cup\rho_2|\leq|\rho_3\cup\rho_4|$,
    \item\label{sass3} if $|\rho_1\cup\rho_2|=n$, then $|\rho_3\cup\rho_4|\geq2^n$, and
    \item\label{sass4} $G$ has a normal Sylow $p$-subgroup for exactly one prime $p\in\rho_3$.
\end{enumerate}
\end{proposition}

Next we state the main theorem from \cite{BL}. It gives rise to two families of graphs that satisfy P\'alfy's condition. We will concern ourselves with one of them, which is our motivation of this paper.

\begin{theorem}\label{markmain}\emph{(\cite{BL})}
Let $\Gamma$ be a graph satisfying P\'alfy's condition with $k\geq5$ vertices. Assume that there exists two vertices $p_1$ and $p_2$ in $\Gamma$ such that
\begin{enumerate}[(i)]
    \item both $p_1$ and $p_2$ are of degree two,
    \item $p_1$ is adjacent to $p_2$, and
    \item $p_1$ and $p_2$ share no common neighbor.
\end{enumerate}
Then $\Gamma$ is not the prime character degree graph of any solvable group.
\end{theorem}

\section{Main Results}

Similar to a direct product, here we construct the graph $\Gamma_{k,t}$. This will yield a family of graphs that we investigate and classify in our Main Theorem. We say that a graph $\Gamma$ is in the family $\Gamma_{k,t}$ if it can be constructed in the following way:

We start with complete graphs $A$ and $B$. Let $A$ have $k$ vertices (arbitrarily labeled $a_1,a_2,\ldots,a_k$), and let $B$ have $t$ vertices (arbitrarily labeled $b_1,b_2,\ldots,b_t$). Without loss of generality, say that $k\geq t$.

Letting $\Gamma_{k,t}:=\Gamma$, we have that $\rho(\Gamma_{k,t})=\rho(A)\cup\rho(B)$, where $\rho(A)$ and $\rho(B)$ are disjoint sets. There is an edge between vertices $p$ and $q$ in $\Gamma_{k,t}$ if any of the following are satisfied:
\begin{enumerate}[(i)]
    \item $p,q\in\rho(A)$,
    \item $p,q\in\rho(B)$, or
    \item $p=a_i\in\rho(A)$ and $q=b_i\in\rho(B)$ for some $1\leq i\leq t$.
\end{enumerate}

Notice how this construction generates a graph that satisfies P\'alfy's condition and is at most diameter two. Moreover, the graphs are symmetric (that is, $\Gamma_{k,t}=\Gamma_{t,k}$). We now proceed towards proving the propositions needed for the Main Theorem.

\begin{proposition}\label{early}
Let $k\geq1$. The graph $\Gamma_{k,1}$ occurs as $\Delta(G)$ for some solvable group $G$.
\end{proposition}
\begin{proof}
Observe that $\Gamma_{1,1}$ occurs by Lemma \ref{complete}. For $k\geq2$, we consider the following: let $A$ be a graph consisting of two complete, connected components: one having an isolated singleton, and the other having $k-1$ vertices. Let $B$ be a singleton. Notice that $A$ occurs as $\Delta(G_A)$ for some solvable group $G_A$ by Lemma \ref{crazy}, and $B$ occurs as $\Delta(G_B)$ for some solvable group $G_B$ by Lemma \ref{complete}. Therefore, the direct product of $A$ and $B$ occurs as $\Delta(G)$ for some solvable group $G$. Finally, observe that $\Gamma_{k,1}$ is precisely the direct product of $A$ and $B$.
\end{proof}

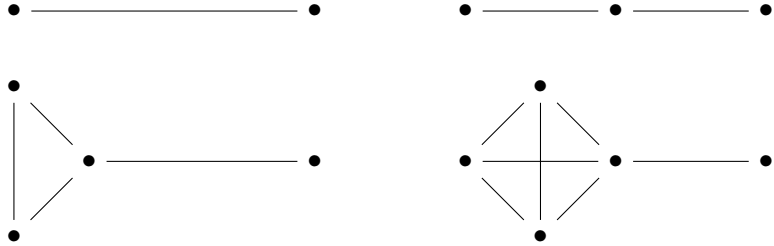
\begin{figure}[H]
    \centering
$
\begin{tikzpicture}[scale=2]
\node (1a) at (0,0) {$\bullet$};
\node (1b) at (2,0) {$\bullet$};
\path[font=\small,>=angle 90]
(1a) edge node [above] {$ $} (1b);
\node (2a) at (3,0) {$\bullet$};
\node (2b) at (4,0) {$\bullet$};
\node (2c) at (5,0) {$\bullet$};
\path[font=\small,>=angle 90]
(2a) edge node [right] {$ $} (2b)
(2b) edge node [above] {$ $} (2c);
\node (3a) at (.5,-1) {$\bullet$};
\node (3b) at (0,-1.5) {$\bullet$};
\node (3c) at (0,-.5) {$\bullet$};
\node (3d) at (2,-1) {$\bullet$};
\path[font=\small,>=angle 90]
(3a) edge node [right] {$ $} (3b)
(3a) edge node [above] {$ $} (3c)
(3b) edge node [above] {$ $} (3c)
(3a) edge node [above] {$ $} (3d);
\node (4a) at (3,-1) {$\bullet$};
\node (4b) at (3.5,-1.5) {$\bullet$};
\node (4c) at (3.5,-.5) {$\bullet$};
\node (4d) at (4,-1) {$\bullet$};
\node (4f) at (5,-1) {$\bullet$};
\path[font=\small,>=angle 90]
(4a) edge node [right] {$ $} (4b)
(4a) edge node [above] {$ $} (4c)
(4a) edge node [above] {$ $} (4d)
(4b) edge node [above] {$ $} (4c)
(4b) edge node [above] {$ $} (4d)
(4c) edge node [right] {$ $} (4d)
(4d) edge node [above] {$ $} (4f);
\end{tikzpicture}
$    
    \caption{Graphs in Proposition \ref{early}}
    \label{figk1}
\end{figure}

It was shown in \cite{H} that $\Gamma_{2,2}$, represented in Figure \ref{fig22}, occurs as $\Delta(G)$ for some solvable group $G$. It is in fact a direct product. The rest of the cases for $t=2$ can be extrapolated from Theorem \ref{markmain}.
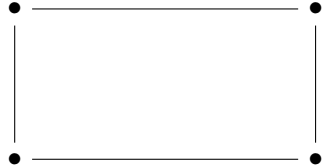
\begin{figure}[H]
    \centering
$
\begin{tikzpicture}[scale=2]
\node (a) at (0,0) {$\bullet$};
\node (b) at (0,1) {$\bullet$};
\node (c) at (2,0) {$\bullet$};
\node (d) at (2,1) {$\bullet$};
\path[font=\small,>=angle 90]
(a) edge node [above] {$ $} (b)
(c) edge node [above] {$ $} (d)
(a) edge node [above] {$ $} (c)
(b) edge node [above] {$ $} (d);
\end{tikzpicture}
$
    \caption{The Graph $\Gamma_{2,2}$}
    \label{fig22}
\end{figure}

\begin{proposition}\label{bissresult}\emph{(\cite{BL})}
Let $k\geq3$. The graph $\Gamma_{k,2}$ is not the prime character degree graph of any solvable group.
\end{proposition}
\begin{proof}
Observe that Theorem \ref{markmain} yields two families, one of which is $\Gamma_{k,2}$.
\end{proof}

\begin{figure}[H]
    \centering
$
\begin{tikzpicture}[scale=2]
\node (1a) at (0,.5) {$\bullet$};
\node (1b) at (.5,0) {$\bullet$};
\node (1c) at (.5,1) {$\bullet$};
\node (1d) at (2,0) {$\bullet$};
\node (1e) at (2,1) {$\bullet$};
\path[font=\small,>=angle 90]
(1a) edge node [right] {$ $} (1b)
(1a) edge node [above] {$ $} (1c)
(1b) edge node [above] {$ $} (1c)
(1d) edge node [above] {$ $} (1e)
(1b) edge node [above] {$ $} (1d)
(1c) edge node [above] {$ $} (1e);
\node (2a) at (3,.5) {$\bullet$};
\node (2b) at (3.5,0) {$\bullet$};
\node (2c) at (3.5,1) {$\bullet$};
\node (2d) at (4,.5) {$\bullet$};
\node (2f) at (5,0) {$\bullet$};
\node (2g) at (5,1) {$\bullet$};
\path[font=\small,>=angle 90]
(2a) edge node [right] {$ $} (2b)
(2a) edge node [above] {$ $} (2c)
(2a) edge node [above] {$ $} (2d)
(2b) edge node [above] {$ $} (2c)
(2b) edge node [above] {$ $} (2d)
(2c) edge node [right] {$ $} (2d)
(2f) edge node [above] {$ $} (2g)
(2b) edge node [above] {$ $} (2f)
(2c) edge node [above] {$ $} (2g);
\node (3i) at (0,-1.25) {$\bullet$};
\node (3a) at (0,-.75) {$\bullet$};
\node (3b) at (.5,-1.5) {$\bullet$};
\node (3c) at (1,-1) {$\bullet$};
\node (3d) at (.5,-.5) {$\bullet$};
\node (3e) at (2,-1.5) {$\bullet$};
\node (3g) at (2,-.5) {$\bullet$};
\path[font=\small,>=angle 90]
(3i) edge node [right] {$ $} (3a)
(3i) edge node [above] {$ $} (3b)
(3i) edge node [above] {$ $} (3c)
(3i) edge node [above] {$ $} (3d)
(3a) edge node [right] {$ $} (3b)
(3a) edge node [above] {$ $} (3c)
(3a) edge node [above] {$ $} (3d)
(3b) edge node [above] {$ $} (3c)
(3b) edge node [above] {$ $} (3d)
(3c) edge node [right] {$ $} (3d)
(3e) edge node [above] {$ $} (3g)
(3b) edge node [above] {$ $} (3e)
(3d) edge node [above] {$ $} (3g);
\node (4i) at (3.25,-1.5) {$\bullet$};
\node (4a) at (3.25,-.5) {$\bullet$};
\node (4b) at (3.75,-1.5) {$\bullet$};
\node (4c) at (3,-1) {$\bullet$};
\node (4h) at (4,-1) {$\bullet$};
\node (4d) at (3.75,-.5) {$\bullet$};
\node (4e) at (5,-1.5) {$\bullet$};
\node (4g) at (5,-.5) {$\bullet$};
\path[font=\small,>=angle 90]
(4h) edge node [right] {$ $} (4a)
(4h) edge node [above] {$ $} (4b)
(4h) edge node [above] {$ $} (4c)
(4h) edge node [above] {$ $} (4d)
(4h) edge node [right] {$ $} (4i)
(4i) edge node [right] {$ $} (4a)
(4i) edge node [above] {$ $} (4b)
(4i) edge node [above] {$ $} (4c)
(4i) edge node [above] {$ $} (4d)
(4a) edge node [right] {$ $} (4b)
(4a) edge node [above] {$ $} (4c)
(4a) edge node [above] {$ $} (4d)
(4b) edge node [above] {$ $} (4c)
(4b) edge node [above] {$ $} (4d)
(4c) edge node [right] {$ $} (4d)
(4e) edge node [above] {$ $} (4g)
(4b) edge node [above] {$ $} (4e)
(4d) edge node [above] {$ $} (4g);
\end{tikzpicture}
$
    \caption{Graphs in Proposition \ref{bissresult}}
    \label{figk2}
\end{figure}
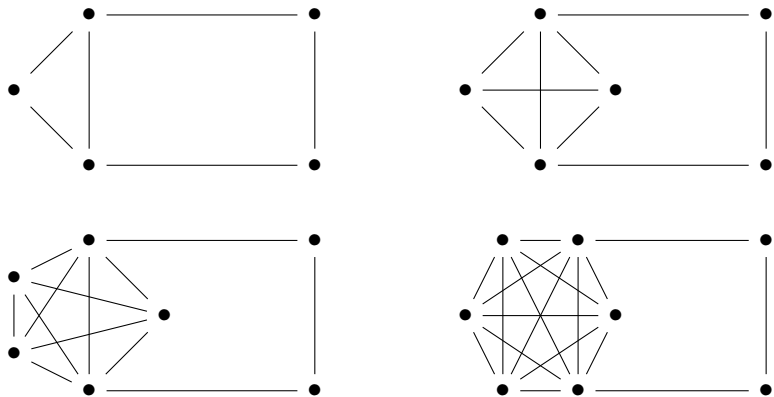

\begin{proposition}\label{main}
Let $k\geq3$. The graph $\Gamma_{k,3}$ is not the prime character degree graph of any solvable group.
\end{proposition}
\begin{proof}
We proceed by induction on $k$. As our base case, we will show that $\Gamma_{3,3}$ does not occur as the prime character degree graph of any solvable group. We label the vertices $a_i$ and $b_i$ for $1\leq i\leq3$, where the $a$'s and $b$'s each form a complete graph, with edges between $a_i$ and $b_i$. For the sake of contradiction, suppose that the solvable group $G$ is a counterexample with $|G|$ minimal such that $\Delta(G)=\Gamma_{3,3}$.

We will show that $a_1$ is admissible. Removing the edge between $a_1$ and $a_i$ (with $1\leq i\leq3$ and $i\neq1$) violates P\'alfy's condition with $a_1$, $a_i$, and $b_j$, where $1\leq j\leq3$, $j\neq1$, and $j\neq i$. Removing the edge between $a_1$ and $b_1$ reduces to a graph of diameter three. This graph cannot occur since $|\rho_3|=2$, which contradicts Proposition \ref{sassresult}\eqref{sass1}. Removing the vertex $a_1$ with all incident edges yields the graph $\Gamma_{3,2}$, which does not occur by Proposition \ref{bissresult}. Thus, $a_1$ is admissible. By symmetry, we note that every remaining vertex is also admissible. Hence, $\Gamma_{3,3}$ cannot occur by Lemma \ref{mark}.

As the inductive hypothesis, we suppose that $\Gamma_{k-1,3}$ does not occur as the prime character degree graph of any solvable group. We now proceed to the inductive step. In order to see that $\Gamma_{k,3}$ also does not occur as the prime character degree graph of any solvable group, one can again show that every vertex is admissible. This follows a similar argument as above, and once again relies on Proposition \ref{sassresult}\eqref{sass1}.
\end{proof}

\begin{figure}[H]
    \centering
$
\begin{tikzpicture}[scale=2]
\node (1a) at (0,1) {$\bullet$};
\node (1b) at (0,0) {$\bullet$};
\node (1c) at (.5,.5) {$\bullet$};
\node (1d) at (2,1) {$\bullet$};
\node (1e) at (2,0) {$\bullet$};
\node (1f) at (1.5,.5) {$\bullet$};
\path[font=\small,>=angle 90]
(1a) edge node [right] {$ $} (1b)
(1a) edge node [above] {$ $} (1c)
(1b) edge node [above] {$ $} (1c)
(1d) edge node [above] {$ $} (1e)
(1d) edge node [above] {$ $} (1f)
(1e) edge node [right] {$ $} (1f)
(1a) edge node [above] {$ $} (1d)
(1b) edge node [above] {$ $} (1e)
(1c) edge node [above] {$ $} (1f);
\node (2a) at (3,.5) {$\bullet$};
\node (2b) at (3.5,0) {$\bullet$};
\node (2c) at (3.5,1) {$\bullet$};
\node (2d) at (4,.5) {$\bullet$};
\node (2e) at (4.5,.5) {$\bullet$};
\node (2f) at (5,0) {$\bullet$};
\node (2g) at (5,1) {$\bullet$};
\path[font=\small,>=angle 90]
(2a) edge node [right] {$ $} (2b)
(2a) edge node [above] {$ $} (2c)
(2a) edge node [above] {$ $} (2d)
(2b) edge node [above] {$ $} (2c)
(2b) edge node [above] {$ $} (2d)
(2c) edge node [right] {$ $} (2d)
(2e) edge node [above] {$ $} (2f)
(2e) edge node [above] {$ $} (2g)
(2f) edge node [above] {$ $} (2g)
(2b) edge node [above] {$ $} (2f)
(2c) edge node [above] {$ $} (2g)
(2d) edge node [above] {$ $} (2e);
\node (3i) at (0,-.75) {$\bullet$};
\node (3a) at (0,-1.25) {$\bullet$};
\node (3b) at (.5,-1.5) {$\bullet$};
\node (3c) at (1,-1) {$\bullet$};
\node (3d) at (.5,-.5) {$\bullet$};
\node (3e) at (2,-1.5) {$\bullet$};
\node (3f) at (1.5,-1) {$\bullet$};
\node (3g) at (2,-.5) {$\bullet$};
\path[font=\small,>=angle 90]
(3i) edge node [right] {$ $} (3a)
(3i) edge node [above] {$ $} (3b)
(3i) edge node [above] {$ $} (3c)
(3i) edge node [above] {$ $} (3d)
(3a) edge node [right] {$ $} (3b)
(3a) edge node [above] {$ $} (3c)
(3a) edge node [above] {$ $} (3d)
(3b) edge node [above] {$ $} (3c)
(3b) edge node [above] {$ $} (3d)
(3c) edge node [right] {$ $} (3d)
(3e) edge node [above] {$ $} (3f)
(3e) edge node [above] {$ $} (3g)
(3f) edge node [right] {$ $} (3g)
(3b) edge node [above] {$ $} (3e)
(3c) edge node [above] {$ $} (3f)
(3d) edge node [above] {$ $} (3g);
\node (4a) at (3,-1) {$\bullet$};
\node (4b) at (3.25,-1.5) {$\bullet$};
\node (4c) at (3.25,-.5) {$\bullet$};
\node (4d) at (3.75,-.5) {$\bullet$};
\node (4e) at (3.75,-1.5) {$\bullet$};
\node (4f) at (4,-1) {$\bullet$};
\node (4g) at (5,-1.5) {$\bullet$};
\node (4h) at (4.5,-1) {$\bullet$};
\node (4i) at (5,-.5) {$\bullet$};
\path[font=\small,>=angle 90]
(4a) edge node [right] {$ $} (4b)
(4a) edge node [above] {$ $} (4c)
(4a) edge node [above] {$ $} (4d)
(4a) edge node [above] {$ $} (4e)
(4a) edge node [right] {$ $} (4f)
(4b) edge node [above] {$ $} (4c)
(4b) edge node [above] {$ $} (4d)
(4b) edge node [above] {$ $} (4e)
(4b) edge node [above] {$ $} (4f)
(4c) edge node [right] {$ $} (4d)
(4c) edge node [above] {$ $} (4e)
(4c) edge node [above] {$ $} (4f)
(4d) edge node [above] {$ $} (4e)
(4d) edge node [above] {$ $} (4f)
(4e) edge node [right] {$ $} (4f)
(4g) edge node [above] {$ $} (4h)
(4g) edge node [above] {$ $} (4i)
(4h) edge node [right] {$ $} (4i)
(4d) edge node [above] {$ $} (4i)
(4e) edge node [above] {$ $} (4g)
(4f) edge node [above] {$ $} (4h);
\end{tikzpicture}
$
    \caption{Graphs in Proposition \ref{main}}
    \label{figk3}
\end{figure}
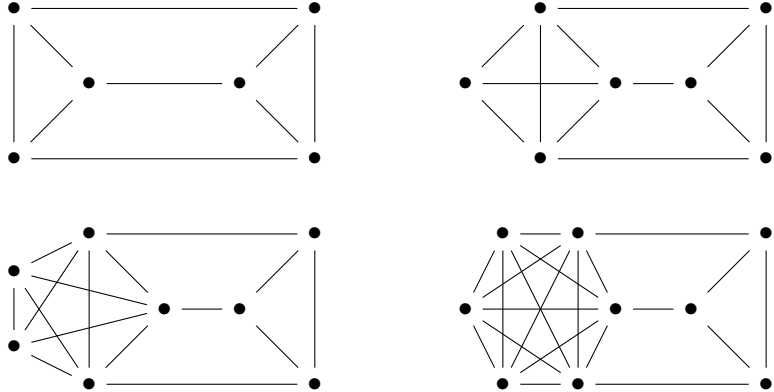

\begin{proposition}\label{huge}
Let $k\geq4$ and $t\geq4$. The graph $\Gamma_{k,t}$ is not the prime character degree graph of any solvable group.
\end{proposition}
\begin{proof}
We proceed by induction on both $k$ and $t$. For our base case, we will show that $\Gamma_{k,4}$ does not occur as the prime character degree graph of any solvable group. To do this, one must first consider $\Gamma_{4,4}$ and then induct on $k$. To see that $\Gamma_{4,4}$ does not occur, we start by labeling the vertices $a_i$ and $b_i$ for $1\leq i\leq4$, where the $a$'s and $b$'s each form a complete graph, with edges between $a_i$ and $b_i$. For the sake of contradiction, suppose that the solvable group $G$ is a counterexample with $|G|$ minimal such that $\Delta(G)=\Gamma_{4,4}$.

We will show that $a_1$ is admissible. Removing the edge between $a_1$ and $a_i$ (with $1\leq i\leq 4$ and $i\neq1$) violates P\'alfy's condition with $a_1$, $a_i$, and $b_j$, where $1\leq j\leq4$, $j\neq1$, and $j\neq i$. Removing the edge between between $a_1$ and $b_1$ yields a graph of diameter three. Notice that $|\rho_1\cup\rho_2|=4$ and $|\rho_3\cup\rho_4|=4$, which contradicts Proposition \ref{sassresult}\eqref{sass3}. Removing the vertex $a_1$ and all incident edges yields the graph $\Gamma_{4,3}$, which does not occur by Proposition \ref{main}. Thus $a_1$ is admissible. By symmetry, notice that every remaining vertex is also admissible. Hence $\Gamma_{4,4}$ cannot occur by Lemma \ref{mark}.

As the inductive hypothesis, we assume that $\Gamma_{k-1,4}$ does not occur as the prime character degree graph of any solvable group. To see that $\Gamma_{k,4}$ does not occur as the prime character degree graph of any solvable group, we start by labeling the vertices $a_i$ and $b_j$ for $1\leq i\leq k$ and $1\leq j\leq4$, where the $a$'s and $b$'s each form a complete graph, with edges between $a_j$ and $b_j$. For the sake of contradiction, suppose that the solvable group $G$ is a counterexample with $|G|$ minimal such that $\Delta(G)=\Gamma_{k,4}$.

First we show that $a_1$ is admissible. Removing the edge between $a_1$ and $a_i$ (where $1\leq i\leq k$ and $i\neq1$) violates P\'alfy's condition with $a_1$, $a_i$, and $b_j$, where $1\leq j\leq4$, $j\neq1$, and $j\neq i$. Removing the edge between $a_1$ and $b_1$ yields a graph of diameter three (which we denote by $\Delta(H)$ for some solvable group $H$). Notice that Proposition \ref{sassresult}\eqref{sass2} forces $\rho_3=\{a_2,a_3,a_4\}$, and by Proposition \ref{sassresult}\eqref{sass4}, we know that $H$ has a normal Sylow $p$-subgroup for exactly one prime $p\in\rho_3$. Without loss of generality, let $p=a_2$, and call the normal subgroup $P$. By Lemma \ref{lewisgraph}, we know that $\rho(H/P')=\rho(H)\setminus\{p\}$. Observe that $\Delta(H/P')$ is a connected subgraph obtained from $\Delta(H)$ by removing the vertex $a_2$ and all incident edges, and possibly edges between vertices adjacent to $a_2$. However, in $\Delta(H/P')$, we have that $|\rho_3|=2$, which contradicts Proposition \ref{sassresult}\eqref{sass1}. Thus the edge between $a_1$ and $b_1$ cannot be lost. Removing the vertex $a_1$ and all incident edges yields a graph of diameter three. By following an argument similar to the above, we can again get that $|\rho_3|=2$, which cannot happen. Thus $a_1$ is admissible. By symmetry, $a_2$, $a_3$, and $a_4$ are admissible as well.

Following similar arguments, one can show that $b_j$ is admissible for $1\leq j\leq4$, as well as $a_i$ for $5\leq i\leq k$. Thus, all vertices are admissible, and we get our result from Lemma \ref{mark}.

Finally we verify that $\Gamma_{k,t}$ does not occur. The inductive hypothesis requires that $\Gamma_{k,t-1}$ does not occur as the prime character degree graph of any solvable group. We then need to show $\Gamma_{k,t}$ does not occur, which requires induction. First we will verify that $\Gamma_{t,t}$ does not occur by following a similar argument as above, mimicking the proof for $\Gamma_{4,4}$. Then we suppose that $\Gamma_{k-1,t}$ does not occur as the prime character degree graph of any solvable group. Showing $\Gamma_{k,t}$ does not occur employs the same argument as for $\Gamma_{k,4}$, using Proposition \ref{sassresult}\eqref{sass4} and Lemma \ref{lewisgraph} multiple times before getting a contradiction with Proposition \ref{sassresult}\eqref{sass1}.
\end{proof}

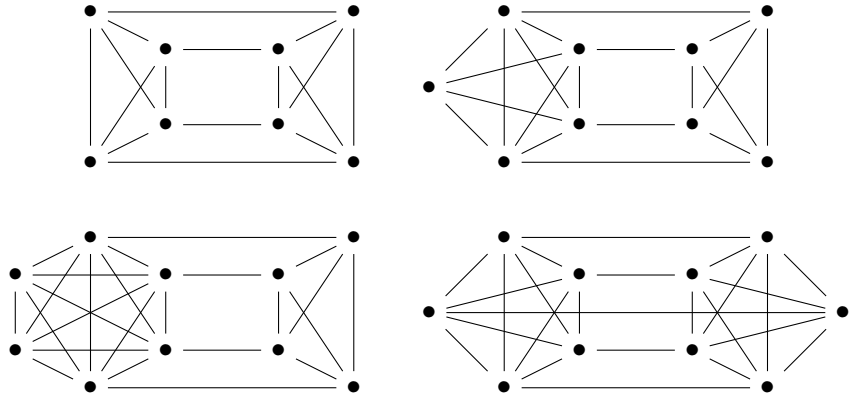
\begin{figure}[H]
    \centering
$
\begin{tikzpicture}[scale=2]
\node (1a) at (0,1) {$\bullet$};
\node (1b) at (0,0) {$\bullet$};
\node (1c) at (.5,.75) {$\bullet$};
\node (1d) at (.5,.25) {$\bullet$};
\node (1e) at (1.25,.75) {$\bullet$};
\node (1f) at (1.25,.25) {$\bullet$};
\node (1g) at (1.75,1) {$\bullet$};
\node (1h) at (1.75,0) {$\bullet$};
\path[font=\small,>=angle 90]
(1a) edge node [right] {$ $} (1b)
(1a) edge node [above] {$ $} (1c)
(1a) edge node [above] {$ $} (1d)
(1b) edge node [above] {$ $} (1c)
(1b) edge node [above] {$ $} (1d)
(1c) edge node [right] {$ $} (1d)
(1e) edge node [above] {$ $} (1f)
(1e) edge node [above] {$ $} (1g)
(1e) edge node [above] {$ $} (1h)
(1f) edge node [above] {$ $} (1g)
(1f) edge node [above] {$ $} (1h)
(1g) edge node [above] {$ $} (1h)
(1a) edge node [above] {$ $} (1g)
(1b) edge node [above] {$ $} (1h)
(1c) edge node [above] {$ $} (1e)
(1d) edge node [above] {$ $} (1f);
\node (2i) at (2.25,.5) {$\bullet$};
\node (2a) at (2.75,1) {$\bullet$};
\node (2b) at (2.75,0) {$\bullet$};
\node (2c) at (3.25,.75) {$\bullet$};
\node (2d) at (3.25,.25) {$\bullet$};
\node (2e) at (4,.75) {$\bullet$};
\node (2f) at (4,.25) {$\bullet$};
\node (2g) at (4.5,1) {$\bullet$};
\node (2h) at (4.5,0) {$\bullet$};
\path[font=\small,>=angle 90]
(2a) edge node [right] {$ $} (2b)
(2a) edge node [above] {$ $} (2c)
(2a) edge node [above] {$ $} (2d)
(2b) edge node [above] {$ $} (2c)
(2b) edge node [above] {$ $} (2d)
(2c) edge node [right] {$ $} (2d)
(2e) edge node [above] {$ $} (2f)
(2e) edge node [above] {$ $} (2g)
(2e) edge node [above] {$ $} (2h)
(2f) edge node [above] {$ $} (2g)
(2f) edge node [above] {$ $} (2h)
(2g) edge node [above] {$ $} (2h)
(2a) edge node [above] {$ $} (2g)
(2b) edge node [above] {$ $} (2h)
(2c) edge node [above] {$ $} (2e)
(2d) edge node [above] {$ $} (2f)
(2i) edge node [above] {$ $} (2a)
(2i) edge node [above] {$ $} (2b)
(2i) edge node [above] {$ $} (2c)
(2i) edge node [above] {$ $} (2d);
\node (3i) at (-.5,-.75) {$\bullet$};
\node (3j) at (-.5,-1.25) {$\bullet$};
\node (3a) at (0,-.5) {$\bullet$};
\node (3b) at (0,-1.5) {$\bullet$};
\node (3c) at (.5,-.75) {$\bullet$};
\node (3d) at (.5,-1.25) {$\bullet$};
\node (3e) at (1.25,-.75) {$\bullet$};
\node (3f) at (1.25,-1.25) {$\bullet$};
\node (3g) at (1.75,-.5) {$\bullet$};
\node (3h) at (1.75,-1.5) {$\bullet$};
\path[font=\small,>=angle 90]
(3i) edge node [right] {$ $} (3j)
(3i) edge node [right] {$ $} (3a)
(3i) edge node [right] {$ $} (3b)
(3i) edge node [right] {$ $} (3c)
(3i) edge node [right] {$ $} (3d)
(3j) edge node [above] {$ $} (3a)
(3j) edge node [above] {$ $} (3b)
(3j) edge node [above] {$ $} (3c)
(3j) edge node [above] {$ $} (3d)
(3a) edge node [right] {$ $} (3b)
(3a) edge node [above] {$ $} (3c)
(3a) edge node [above] {$ $} (3d)
(3b) edge node [above] {$ $} (3c)
(3b) edge node [above] {$ $} (3d)
(3c) edge node [right] {$ $} (3d)
(3e) edge node [above] {$ $} (3f)
(3e) edge node [above] {$ $} (3g)
(3e) edge node [above] {$ $} (3h)
(3f) edge node [above] {$ $} (3g)
(3f) edge node [above] {$ $} (3h)
(3g) edge node [above] {$ $} (3h)
(3a) edge node [above] {$ $} (3g)
(3b) edge node [above] {$ $} (3h)
(3c) edge node [above] {$ $} (3e)
(3d) edge node [above] {$ $} (3f);
\node (4i) at (2.25,-1) {$\bullet$};
\node (4j) at (5,-1) {$\bullet$};
\node (4a) at (2.75,-.5) {$\bullet$};
\node (4b) at (2.75,-1.5) {$\bullet$};
\node (4c) at (3.25,-.75) {$\bullet$};
\node (4d) at (3.25,-1.25) {$\bullet$};
\node (4e) at (4,-.75) {$\bullet$};
\node (4f) at (4,-1.25) {$\bullet$};
\node (4g) at (4.5,-.5) {$\bullet$};
\node (4h) at (4.5,-1.5) {$\bullet$};
\path[font=\small,>=angle 90]
(4a) edge node [right] {$ $} (4b)
(4a) edge node [above] {$ $} (4c)
(4a) edge node [above] {$ $} (4d)
(4b) edge node [above] {$ $} (4c)
(4b) edge node [above] {$ $} (4d)
(4c) edge node [right] {$ $} (4d)
(4e) edge node [above] {$ $} (4f)
(4e) edge node [above] {$ $} (4g)
(4e) edge node [above] {$ $} (4h)
(4f) edge node [above] {$ $} (4g)
(4f) edge node [above] {$ $} (4h)
(4g) edge node [above] {$ $} (4h)
(4a) edge node [above] {$ $} (4g)
(4b) edge node [above] {$ $} (4h)
(4c) edge node [above] {$ $} (4e)
(4d) edge node [above] {$ $} (4f)
(4i) edge node [above] {$ $} (4a)
(4i) edge node [above] {$ $} (4b)
(4i) edge node [above] {$ $} (4c)
(4i) edge node [above] {$ $} (4d)
(4j) edge node [above] {$ $} (4e)
(4j) edge node [above] {$ $} (4f)
(4j) edge node [above] {$ $} (4g)
(4j) edge node [above] {$ $} (4h)
(4i) edge node [above] {$ $} (4j);
\end{tikzpicture}
$
    \caption{Graphs in Proposition \ref{huge}}
    \label{figk4}
\end{figure}

As consequence of Propositions \ref{early}, \ref{bissresult}, \ref{main}, and \ref{huge}, our Main Theorem is proved. We state it again below:

\begin{theorem}
The graph $\Gamma_{k,t}$ occurs as the prime character degree graph of a solvable group precisely when $k=1$, or $t=1$, or $k=t=2$.
\end{theorem}

\begin{corollary}
Let $k\geq2$ and $t\geq2$. No connected proper subgraph of $\Gamma_{k,t}$ with the same vertex set is the prime character degree graph of any solvable group.
\end{corollary}

The case for the disconnected subgraph can also be considered. The group structure of two connected components is fully outlined in \cite{L2}.

\end{document}